\documentclass[10pt]{article}
\usepackage{cite}
\usepackage{mathpazo}
\usepackage{amsmath}
\usepackage{amsthm}
\usepackage{amssymb}
\usepackage{url}
\usepackage{latexsym}
\usepackage[xcolor=pst]{pstricks}
\usepackage{graphicx, pst-plot, pst-node, pst-text, pst-tree}
\usepackage{titlefoot}
\usepackage[small,pagestyles]{titlesec}
\usepackage{units} 
\usepackage[small,it]{caption}

\usepackage{color}
\definecolor{lightgray}{rgb}{0.8, 0.8, 0.8}
\definecolor{darkgray}{rgb}{0.7, 0.7, 0.7}
\definecolor{darkblue}{rgb}{0, 0, .4}

\usepackage[bookmarks]{hyperref}
\hypersetup{
        colorlinks=true,
        linkcolor=darkblue,
        anchorcolor=darkblue,
        citecolor=darkblue,
        urlcolor=darkblue,
        pdfpagemode=UseThumbs,
        pdftitle={Linear clique-width for hereditary classes of cographs},
        pdfsubject={Combinatorics},
        pdfauthor={Brignall, Korpelainen, and Vatter},
}

\newcounter{todocounter}

\newcommand{\note}[1]{
	\marginpar{{\footnotesize\it #1}}
}

\newcommand{\minisec}[1]{\medskip\noindent{\bf #1.}}

\theoremstyle{plain}
\newtheorem{theorem}{Theorem}[section]
\newtheorem{proposition}[theorem]{Proposition}

\newtheorem{lemma}[theorem]{Lemma}
\newtheorem{corollary}[theorem]{Corollary}
\newtheorem{conjecture}[theorem]{Conjecture}

\theoremstyle{definition}

\makeatletter
\newfont{\footsc}{cmcsc10 at 8truept}
\newfont{\footbf}{cmbx10 at 8truept}
\newfont{\footrm}{cmr10 at 10truept}
\renewenvironment{abstract}%
                {
                  \begin{list}{}%
                     {\setlength{\rightmargin}{1in}%
                      \setlength{\leftmargin}{1in}}%
                   \item[]\ignorespaces\begin{small}}%
                 {\end{small}\unskip\end{list}}

\newcommand{\edgeless}{\overline{\mathcal{K}}}

\newcommand{\st}{\::\:}
\newcommand{\Av}{\operatorname{Av}}
\newcommand{\Free}{\operatorname{Free}}

\newcommand{\C}{\mathcal{C}}
\newcommand{\D}{\mathcal{D}}

\newcommand{\J}{\mathcal{J}}

\newcommand{\U}{\mathcal{U}}

\newcommand{\lcw}{\operatorname{lcw}}
\newcommand{\cw}{\operatorname{cw}}

\datefoot{January 25, 2016}
\amssubj{05C78, 68R10, 05C85}

\newpagestyle{main}[\small]{
        \headrule
        \sethead[\usepage][][]
        {\sc Linear Clique-Width for Hereditary Classes of Cographs}{}{\usepage}}

\title{\sc Linear Clique-Width for Hereditary Classes of Cographs}
\author{%
Robert Brignall\footnote{All three authors were partially supported by EPSRC via the grant EP/J006130/1.}\\[-0.25ex]
\small Department of Mathematics and Statistics\\[-0.5ex]
\small The Open University\\[-0.5ex]
\small Milton Keynes, UK
\\[1.5ex]
\and
Nicholas Korpelainen\footnotemark[\value{footnote}]\\[-0.25ex]
\small Mathematics Department\\[-0.5ex]
\small University of Derby\\[-0.5ex]
\small Derby, UK
\\[1.5ex]
\and
Vincent Vatter\footnotemark[\value{footnote}]\footnote{Vatter's research was also partially supported by the National Security Agency under Grant Number H98230-12-1-0207 and the National Science Foundation under Grant Number DMS-1301692.  The United States Government is authorized to reproduce and distribute reprints not-withstanding any copyright notation herein.}\\[-0.25ex]
\small Department of Mathematics\\[-0.5ex]
\small University of Florida\\[-0.5ex]
\small Gainesville, Florida USA
\\[-1.5ex]
}

\titleformat{\section}
        {\large\sc}
        {\thesection.}{1em}{}

\date{}

\begin{document}
\maketitle

\pagestyle{main}

\begin{abstract}
The class of cographs is known to have unbounded linear clique-width. We prove that a hereditary class of cographs has bounded linear clique-width if and only if it does not contain all quasi-threshold graphs or their complements. The proof borrows ideas from the enumeration of permutation classes.
\end{abstract}

\section{Introduction}

A variety of measures of the complexity of graph classes have been introduced and have proved useful for algorithmic problems~\cite{corneil:on-the-relation:,courcelle:the-expression-:,courcelle:upper-bounds-to:,courcelle:linear-time-sol:,espelage:how-to-solve-np:,kaminski:recent-developm:,kobler:edge-dominating:,oum:approximating-c:,wanke:k-nlc-graphs-an:,bui-xuan:boolean-width-o:}. We are concerned with \emph{clique-width}, introduced by Courcelle, Engelfriet, and Rozenberg~\cite{courcelle:handle-rewritin:}, and more pertinently, the linear version of this parameter, due to Lozin and Rautenbach~\cite{lozin:the-relative:techreport}, and studied further by Gurski and Wanke~\cite{gurski:on-the-relation:}.

The \emph{linear clique-width} of a graph $G$, $\lcw(G)$, is the size of the smallest alphabet $\Sigma$ such that $G$ can be constructed by a sequence of the following three operations:
\begin{itemize}
\item add a new vertex labeled by a letter in $\Sigma$,
\item add edges between all vertices labeled $i$ and all vertices labeled $j$ (for $i\neq j$), and
\item give all vertices labeled $i$ the label $j$.
\end{itemize}
Such a sequence of operations is referred to as an \emph{lcw expression} for $G$.  An \emph{efficient} lcw expression for $G$ is one using precisely $\lcw(G)$ labels. The distinction between linear clique-width and clique-width is that clique-width expressions also allow one to take the disjoint union of two previously constructed graphs (along with their labels, which need not be distinct).

The \emph{join} of two vertex-disjoint graphs $G$ and $H$, denoted $G\ast H$, is formed from the disjoint union $G\cup H$ by adding all possible edges with one endpoint in $G$ and the other in $H$. In this context, $G$ and $H$ are referred to as the \emph{join components} of the resulting graph. We consider the following families of graphs in this work. The characterisations we give below can be found in numerous places, for instance in the book \emph{Graph Classes: A Survey}~\cite{brandstadt:graph-classes:-:} by Brandst\"adt, Le, and Spinrad.
\begin{itemize}
\item A graph $G$ is a \emph{threshold graph} if, starting with $K_1$, it can be built by repeatedly taking the disjoint union or join of a threshold graph with $K_1$. This class consists of precisely those graphs which avoid $P_4$, $C_4$, and $2K_2$ as induced subgraphs.
\item A graph $G$ is a \emph{quasi-threshold graph} if, starting with $K_1$, it can be built by repeatedly taking the disjoint union of two quasi-threshold graphs, or the join of a quasi-threshold graph with $K_1$. This class consists of precisely those graphs which avoid $P_4$ and $C_4$ as induced subgraphs, and these graphs are also called \emph{trivially perfect}.
\item A graph $G$ is a \emph{cograph} if, starting with $K_1$, it can be built by repeatedly taking the disjoint union or join of two cographs. This class consists of precisely those graphs which avoid $P_4$ as an induced subgraph.
\end{itemize}
Thus, 
$$
\mbox{threshold graphs}
\subsetneq
\mbox{quasi-threshold graphs}
\subsetneq
\mbox{cographs}.
$$

For a class of graphs $\C$, the linear clique-width of $\C$, $\lcw(\C)$, is defined as the maximum value of $\lcw(G)$ over all graphs in $\C$, if it exists. The class of cographs is precisely the class of graphs with clique-width at most $2$ (see~\cite{courcelle:upper-bounds-to:}), but Gurski and Wanke~\cite{gurski:on-the-relation:} showed that it has unbounded linear clique-width.\footnote{In fact, they showed that cographs have unbounded clique-tree-width, a parameter which lies between clique-width and linear clique-width. It should be noted that the quasi-threshold graphs in fact have clique-tree-width $2$.} On the other hand, threshold graphs have linear clique-width at most $2$. (Gurski~\cite{gurski:characterizatio:} showed that a graph has linear clique-width at most $2$ if and only if it avoids $P_4$, $2K_2$, and $\overline{2P_3}$.)

These facts are often cited in the literature (e.g., in ~\cite{heggernes:graphs-of-linea:,heggernes:a-complete-char:}) to demonstrate that linear clique-width is a significant restriction of clique-width.  We sharpen these results by demarcating the precise boundary between subclasses of cographs with bounded and unbounded linear clique-width:

\begin{theorem}
\label{thm-lcw-main}
A hereditary class of cographs has bounded linear clique-width if and only if it contains neither all quasi-threshold graphs nor the complements of all quasi-threshold graphs.
\end{theorem}

In other words, Theorem~\ref{thm-lcw-main} shows that the quasi-threshold graphs are a minimal class of unbounded linear clique-width, paralleling Lozin's constructions of minimal classes of unbounded clique-width~\cite{lozin:minimal-classes:}.

We prove that the class of quasi-threshold graphs has unbounded linear clique-width in the next section. That the class of their complements also has unbounded linear clique-width follows from the following observation, mentioned in Gurski and Wanke~\cite{gurski:on-the-relation:}.

\begin{proposition}
\label{prop-lcw-complement}
For every graph $G$ we have $\lcw(G)-1\le\lcw(\overline{G})\le\lcw(G)+1$.
\end{proposition}

Proposition~\ref{prop-lcw-complement} follows from the fact that $\lcw(\overline{G})\le\lcw(G)+1$. One way to establish this is to note that every graph $G$ has an lcw expression with at most $\lcw(G)+1$ labels in which every vertex is inserted with a particular label reserved for this purpose (which we will call an \emph{insertion label}), and then immediately connected to all of its neighbors which have already been inserted. Thus by changing these edge operations we can build $\overline{G}$ in the same manner,  also with $\lcw(G)+1$ labels.

Throughout this paper, a set of graphs closed under isomorphism is referred to as a \emph{property}, and a property which is closed downward under the induced subgraph partial order is called \emph{hereditary}. As all properties of graphs considered here are hereditary, we use the term \emph{class} to refer to a hereditary property. We assume that the reader is familiar with basic graph notions and terminology as presented in, for example, Diestel's text~\cite{diestel:graph-theory:}.

\section{Quasi-Threshold Graphs have Unbounded Linear Clique-Width}
\label{sec-unbounded-lcw}

Our proof that the quasi-threshold graphs have unbounded linear clique-width follows from a construction inspired by the work of Gurski and Wanke~\cite{gurski:on-the-relation:}, and requires several preliminary results. The crucial difference between our construction and Gurski and Wanke's is the notion of a sink-free lcw expression, defined shortly.

First note that we can, without loss of generality, restrict our attention to lcw expressions in which labels can only be introduced by adding a new vertex with that label. We need two further basic observations about lcw expressions. In the following, we say that a label is in \emph{use} (when constructing a graph from an lcw expression) if, at that point in time, there is a vertex with that label.

\begin{proposition}
\label{lem-all-in-use}
In the process of building the graph $G$ via any efficient lcw expression, there is a point in the process when all $\lcw(G)$ labels are in use.
\end{proposition}
\begin{proof}
Suppose that there is some efficient lcw expression where at most $\lcw(G)-1$ labels are ever in use simultaneously. Among all such lcw expressions for $G$, consider those which have shortest length, in terms of the number of operations in the expression. Among these shortest expressions, choose the one in which the last label to be introduced is introduced as late in the expression as possible, and denote this label by $\ell$. We achieve a contradiction by modifying the expression so that the first use of $\ell$ occurs even later. To do this, we follow the lcw expression up until the point when the label $\ell$ is first due to be used. At this point, there must be another label, say $k$, which is not currently in use (by our hypotheses). In the rest of the lcw expression, we swap all occurrences of $\ell$ for $k$, pushing the first use of $\ell$ later in the expression (but not increasing the number of operations in the expression), establishing the desired contradiction.
\end{proof}

\begin{proposition}
\label{prop-label-order}
Fix an efficient lcw expression for the graph $G$ and a label $\ell$ used in this expression. Then there is an lcw expression for $G$ in which the vertices are inserted in the same order and where vertices labeled $\ell$ are never relabeled.
\end{proposition}
\begin{proof}
Suppose that the given lcw expression does not preserve $\ell$, and identify the first instance when vertices labeled $\ell$ are relabeled by some other label $k$. Instead, replace this operation by relabeling vertices labeled $k$ to $\ell$, and then swap all subsequent instances of $\ell$ and $k$. Finally, repeat this process until the desired expression is created.
\end{proof}

We say that a label in an lcw expression is a \emph{sink label} if it is never used to create edges nor ever relabeled. In other words, it is a label that can only be given to vertices whose neighborhoods have been completed. We further say that an lcw expression is \emph{sink-free} if no label of this expression is a sink label.

\begin{proposition}
\label{prop-sink-increases}
Suppose that every efficient lcw expression for the graph $G$ is sink-free. Then the graph $H=G\cup G$ satisfies $\lcw(H)=\lcw(G)+1$ and moreover, $H$ has an efficient lcw expression with a sink.
\end{proposition}
\begin{proof}
First we prove that $\lcw(G)+1$ is a lower bound on $\lcw(H)$. Suppose to the contrary that $H$ has an lcw expression with $\lcw(G)$ labels. Denote the vertex sets of the two copies of $G$ in $H$ by disjoint sets $X$ and $Y$. By symmetry we may assume that the first vertex which is inserted lies in $X$. Denote this vertex by $x$ and let $\ell$ denote the label of $x$ when it is inserted. By Proposition~\ref{prop-label-order}, there is an lcw expression for $H$ using $\lcw(G)$ labels in which $x$ is still the first vertex inserted and in which vertices labeled $\ell$ (specifically, the vertex $x$) are never relabeled.

Since the restriction of this lcw expression to $Y$ induces an lcw expression for $G$, there is at least one point during the construction where every label is used to label some vertex in $Y$, by Proposition~\ref{lem-all-in-use}. Since $x$ is labeled $\ell$ at this point of the construction, however, no edges may ever be added between a vertex of $Y$ labeled $\ell$ and any other vertex of $Y$, as this would also create an edge between $x$ and a vertex of $Y$. Therefore, in the restriction of this lcw expression to the vertices of $Y$, $\ell$ is a sink label, contradicting our hypotheses.

It is much easier to show that $\lcw(H)\le\lcw(G)+1$: use any lcw expression of $G$ with $\lcw(G)$ labels to create a copy of $G$, replace all of these labels by a new label, and then repeat the lcw expression of $G$ to create a disjoint copy of $G$. Note that the lcw expression we have built for $H$ has a sink label, namely the new label we introduced.
\end{proof}

\begin{proposition}
\label{prop-increase-makes-sink}
Suppose that the graph $G$ is not edgeless and has an efficient lcw expression with a sink. The graph $H = K_1\ast(G\cup G)$ satisfies $\lcw(H)=\lcw(G)$ and moreover, every efficient lcw expression for $H$ is sink-free.
\end{proposition}
\begin{proof}
First we show that $\lcw(H)=\lcw(G)$. We have $\lcw(H)\ge\lcw(G)$ because linear clique-width is hereditary, so we need only show that $H$ has an lcw expression with $\lcw(G)$ labels. We follow an efficient lcw expression with sink $\ell$ to build one copy of $G$, before relabeling all vertices of this copy with the label $\ell$. We then follow the same lcw expression to build a disjoint copy of $G$, and again relabel all vertices with $\ell$. Since $G$ has an edge, $\lcw(G)\ge 2$, and thus $G$ has at least one non-$\ell$ label. Insert a final vertex of this label and add edges between it and every vertex constructed before.

Now suppose for the sake of contradiction that $H$ has an efficient lcw expression with sink label $\ell$. Denote the vertices of the two copies of $G$ in $H$ by disjoint sets $X$ and $Y$ and let $z$ denote the vertex of $H$ which is adjacent to all of $X\cup Y$. Using Proposition~\ref{lem-all-in-use}, suppose that when this lcw expression is restricted to $X$ (resp., $Y$), the insertion of the vertex $x$ (resp., $y$) marks the first point at which all labels are in use. By symmetry, we may assume that $x$ is inserted before $y$. Thus when $x$ is inserted, there is a vertex $x'\in X$ labeled by $\ell$. Since we assumed that $\ell$ is a sink label in the expression for $H$, $x'$ can never acquire another edge. Therefore $z$ must have been inserted before $x'$, and hence also before $x$. At the point $x$ is inserted, there is a vertex $x''\in X$ with the same label as $z$. Since $x''$ and $z$ will always now have the same label, and $x''$ is inserted before $y$, it is now impossible to add the edge between $y$ and $z$ without adding an edge between $y$ and $x''$, a contradiction.
\end{proof}

Although not necessary for our upcoming construction, note that for $G$ not satisfying the hypotheses of Proposition~\ref{prop-increase-makes-sink}, it is still possible to compute the linear clique-width of $H=K_1\ast(G\cup G)$. If $G$ does not have an edge, then $H=K_1\ast \overline{K}_n$ for some $n$ so $\lcw(H)=2$. In the other case, if all efficient lcw expressions for $G$ are sink-free, then Proposition~\ref{prop-sink-increases} shows that $\lcw(H)\ge\lcw(G\cup G)=\lcw(G)+1$. In fact, $\lcw(H)=\lcw(G)+1$, as can be seen by following the lcw expression given in the proof of Proposition~\ref{prop-sink-increases} to build $G\cup G$, then relabeling every vertex with the same label, inserting a new vertex with a different label, and adding all edges between this vertex and the rest of the graph.

%

We now have the necessary tools to prove one half of Theorem~\ref{thm-lcw-main}.

\begin{corollary}
\label{thm-lcw-main-first-half}
The class of quasi-threshold graphs has unbounded linear clique-width.
\end{corollary}
\begin{proof}
Set $G_1=K_2$, and for $k\ge 1$ define $G_{k+1}$ by
$$
G_{k+1} = K_1\ast((G_k\cup G_k)\cup (G_k\cup G_k)).
$$
Clearly these are all quasi-threshold graphs. By inductively applying Proposition~\ref{prop-sink-increases} to $G_k$, and then inserting the outcome for $G_k\cup G_k$ into Proposition~\ref{prop-increase-makes-sink}, we conclude that $\lcw(G_k) = k+1$ for all $k\ge 1$, proving the result.
\end{proof}

In the proof above, note that we took four copies of $G_k$ to form $G_{k+1}$ merely for convenience; there may well be a construction to prove Corollary~\ref{thm-lcw-main-first-half} which adds fewer vertices at each step.

As already observed, Proposition~\ref{prop-lcw-complement} implies that the class of complements of quasi-threshold graphs also has unbounded linear clique-width.

\section{Structural Notions}

Our proof of the other half of Theorem~\ref{thm-lcw-main} requires the notions of the substitution (or modular) decomposition, well-quasi-ordering, and atomicity. Here we briefly review these notions and some other facts required in our proof.

Let $G$ be a graph on the labeled vertices $\{1,\dots,n\}$ and $(H_1,\dots,H_n)$ a sequence of graphs. The \emph{inflation} of $G$ by $H_1,\dots,H_n$ is the graph on the disjoint union $H_1\cup\cdots\cup H_n$ with additional edges added between all vertices of $H_i$ and all vertices of $H_j$ if and only if $i$ and $j$ are adjacent in $G$. We denote this graph by $G[H_1,\dots,H_n]$. 

Given classes $\C$ and $\D$, the inflation of $\C$ by $\D$, which we denote $\C[\D]$, is the set of all inflations $G[H_1,\dots,H_n]$ where $G$ is a labeling of a graph from $\C$ and $H_1,\dots,H_n\in\D$. The relationship between the substitution decomposition and the (not necessarily linear) clique-width parameter was thoroughly studied by Courcelle and Olariu~\cite{courcelle:upper-bounds-to:}, who established that $\cw(\C[\D])=\max\{\cw(\C),\cw(\D)\}$. For linear clique-width, we observe the following.

\begin{proposition}
\label{prop-inflate-lcw}
For classes $\C$ and $\D$ of graphs with bounded linear clique-width, $\lcw(\C[\D])\le\lcw(\C)+\lcw(\D)$.
\end{proposition}
\begin{proof}
Take any graph $G\in\C$ with $n$ vertices labeled by $\{1,\dots,n\}$, and any sequence $(H_1,\dots,H_n)$ of graphs in $\D$. We want to find an lcw expression for $G[H_1,\dots,H_n]$ with at most $\lcw(G)+\max\{\lcw(H_i)\}$ labels.  We build this expression by following an efficient lcw expression for $G$. Whenever this expression instructs us to insert the vertex $v_i\in G$, we instead use a disjoint set of $\lcw(H_i)$ labels to create a copy of $H_i$. Once this copy of $H_i$ is completed, we relabel all of its vertices to the label given to $v_i$ in the lcw expression of $G$ and continue.
\end{proof}

Letting $\edgeless$ denote the class of all edgeless graphs, note that $\lcw(\C[\edgeless])=\lcw(\C)$; indeed, it is known (see \cite{heggernes:graphs-of-linea:}) that graphs of linear clique-width at most $2$ are precisely the inflations of threshold graphs by $\edgeless$ (as mentioned in the introduction, these are also the graphs which avoid $P_4$, $2K_2$, and $\overline{2P_3}$).

A class of finite graphs is said to be \emph{well-quasi-ordered (wqo)} under the induced subgraph partial order if given any infinite sequence $G_1$, $G_2$, $\dots$ of graphs in the class, we can find indices $i<j$ such that $G_i$ is an induced subgraph of $G_j$. The folklore result below (see Kruskal~\cite{kruskal:the-theory-of-w:}) gives several equivalent conditions. In the statement of this result we say that the subclasses of the class $\C$ satisfy the \emph{descending chain condition} if every infinite descending chain of subclasses
$$
\C\supseteq\D_1\supseteq\D_2\supseteq\cdots
$$
contains a minimum element, that is, in any such sequence there is an index $i$ such that $\D_j=\D_i$ for all $j\ge i$.

\begin{theorem}[folklore]
\label{thm-wqo-folklore}
Let $\C$ be a class of finite graphs. The following are equivalent:
\begin{enumerate}
\item[(a)] $\C$ is wqo under the induced subgraph order,
\item[(b)] $\C$ does not contain an infinite antichain,
\item[(c)] given any infinite sequence $G_1$, $G_2$, $\dots$ of graphs of $\C$, there is an infinite ascending subsequence, i.e., a sequence $G_{i_1}$, $G_{i_2}$, $\dots$ such that each graph in the sequence is an induced subgraph of the next graph in the subsequence,
\item[(d)] the subclasses of $\C$ satisfy the descending chain condition.
\end{enumerate}
\end{theorem}

Cographs are known to be wqo (see Damaschke~\cite{damaschke:induced-subgrap:}), and we make use of both parts (c) and (d) of Theorem~\ref{thm-wqo-folklore} in our proof.

Note that there is no general relation between linear clique-width and well-quasi-ordering: in one direction, cographs are wqo but have unbounded linear clique-width, while the set of all cycles forms an infinite antichain in the induced subgraph partial order, but each cycle has linear clique-width at most $4$. 

Our final concept is atomicity. We say that the class $\C$ is \emph{atomic} if it cannot be written as the union of two proper subclasses.  This condition is equivalent to the \emph{joint embedding property}: for every pair of graphs in $\C$ there is another graph in $\C$ which contains both. (Fra{\"{\i}}ss{\'e}~\cite{fraisse:sur-lextension-:} seems to have been the first to study these notions.) We will need to make use of the combination of atomicity and well-quasi-order via the proposition, which appears to be folklore. As it is relatively unknown, for completeness we give a proof.

\begin{proposition}
Every wqo class of graphs is a finite union of atomic classes.
\end{proposition}
\begin{proof}
Suppose that $\C$ is a wqo class. We recursively build a binary tree whose root is labeled by $\C$. Consider a node labeled by $\D$. If $\D=\D_1\cup\D_2$ for proper subclasses $\D_1$ and $\D_2$ then we add children labeled by $\D_1$ and $\D_2$ to this node. Otherwise $\D$ is atomic and nodes labeled $\D$ in the binary tree are leaves. Because $\C$ is wqo it does not have an infinite descending chain of subclasses by Theorem~\ref{thm-wqo-folklore} (d), so this tree contains no infinite paths and thus is finite by K\"onig's Lemma.  Its leaves give the desired atomic classes.
\end{proof}

\section{The Other Half of Theorem~\ref{thm-lcw-main}}
\label{sec-proof-of-main}

Having shown that the quasi-threshold graphs have unbounded linear clique-width, this section concerns the proof of the other direction of Theorem~\ref{thm-lcw-main}. Our proof, which is inspired by the work of Albert, Atkinson, and Vatter~\cite{albert:subclasses-of-t:} on permutation classes, begins by proving that every infinite subclass $\C$ of cographs either contains all quasi-threshold graphs or their complements, or is contained in the inflation of threshold graphs by a proper subclass $\D\subsetneq\C$. 

If $\C$ were closed under both disjoint unions and joins then it would contain all cographs. Thus $\C$ is either closed under one of these operations or under neither of them. If it is closed under disjoint unions or joins, the following lemma gives what we need. Note that both the class of edgeless graphs and the class of complete graphs are subclasses of threshold graphs.

\begin{lemma}
\label{lem-closed-one}
Let $\C$ be a nonempty subclass of cographs that is closed under either disjoint unions or joins. Either $\C$ is contained in the inflation of the edgeless or complete graphs by a proper subclass $\D\subsetneq\C$, or $\C$ contains all quasi-threshold graphs or all of their complements.
\end{lemma}
\begin{proof}
Pick a subclass $\C$ of cographs, which by symmetry we may assume to be closed under disjoint unions. Let $\C_c$ denote the downward closure of the connected graphs in $\C$. If $\C_c$ is a proper subclass of $\C$ then $\C\subseteq\edgeless[\C_c]$, where $\edgeless$ denotes the class of edgeless graphs (all of which are threshold graphs).

If $\C_c=\C$, then every graph in $\C$ is contained in a connected graph in $\C$. In this case, we prove by induction on $|V(G)|$ that every quasi-threshold graph $G$ is contained in $\C$. The base case $|V(G)|=1$ is trivial, as $\C$ is nonempty. Now suppose that $|V(G)|\ge 2$. If $G$ is a disjoint union then we are done because $\C$ is closed under disjoint unions. Otherwise, $G=H\ast K_1$ for a graph $H$ which lies in $\C$ by induction. Again because $\C$ is closed under disjoint unions, $H\cup K_1\in\C$, and because $\C=\C_c$, $H\cup K_1$ is contained as an induced subgraph in a connected graph in $\C$. Finally, because every connected cograph is a join, we see that $\C$ contains $(H\cup K_1)\ast K_1$, and thus also $G$, as desired.
\end{proof}


The remaining case is that where $\C$ is closed under neither disjoint unions nor joins. Our proof of Theorem~\ref{thm-lcw-main} will seek a minimal counterexample, and since cographs are wqo, we can appeal to Theorem~\ref{thm-wqo-folklore} (d) to establish that a minimal counterexample must be atomic. Consequently, we can restrict the following lemma to atomic subclasses of cographs.

\begin{lemma}
\label{lem-closed-neither}
Let $\C$ be an infinite atomic subclass of cographs. If $\C$ is closed under neither disjoint unions nor joins then it is contained in the inflation of the threshold graphs by a proper subclass $\D\subsetneq\C$.
\end{lemma}
\begin{proof}
Let $\C$ be an atomic subclass of cographs which is closed under neither disjoint unions nor joins.  For every graph $G\in\C$, let $\U_G^+$ denote the set
$$
\U_G^+=\{H\in\C\st G\cup H\in\C\}
$$
of those graphs whose disjoint unions with $G$ lie in $\C$. We further define the set
\begin{eqnarray*}
\U_G^-
&=&
\{H\in \C\st \mbox{for every $H'\in\C$, if $G\cup H'\in\C$ then $H\cup H'\in\C$}\},\\
&=&
\{H\in \C\st H\cup H'\in\C\mbox{ for all } H'\in U_G^+\}
\end{eqnarray*}
of all graphs which are ``at least as unionable'' as $G$.

First note that $\U_G^+$ and $\U_G^-$ are classes. Moreover, if for some graph $G\in\C$ we were to have $\U_G^+=\U_G^-=\C$, then that would mean that for every graph $H\in\U_G^-=\C$ and every graph $H'\in\U_G^+=\C$ their disjoint union $H\cup H'$ lies in $\C$, contradicting our assumption that $\C$ is not closed under disjoint unions. Thus for every graph $G\in\C$, we define the proper subclass
$$
\U_G
=
\left\{\begin{array}{ll}
\U_G^+&\mbox{if $\U_G^+\neq\C$, or}\\
\U_G^-&\mbox{otherwise (in which case $\U_G^-\neq\C$ by the above).}
\end{array}\right.
$$

Observe that passing from $G$ to $\U_G$ reverses inclusions for $\U_G^+$, but not for $\U_G^-$: i.e., if $H$ is an induced subgraph of $G$ then $\U_G^+\subseteq\U_H^+$,  but $\U_H^-\subseteq \U_G^-$. Note also that if $\U_G^+ = \U_H^+$, then we also have $\U_G^- = \U_H^-$.

We first claim that the number of \emph{distinct} classes of the form $\U_G^+$ is finite.  Suppose to the contrary that there are infinitely many distinct classes $\U_{G_1}^+,\U_{G_2}^+,\dots$. By Theorem~\ref{thm-wqo-folklore} (c), the sequence $G_1$, $G_2$, $\dots$ contains an infinite ascending sequence $G_{i_1}$, $G_{i_2}$, $\dots$. Our observation above then shows that
$$
\U_{G_{i_1}}^+\supsetneq \U_{G_{i_2}}^+\supsetneq\cdots.
$$
This, however, contradicts the descending chain condition (Theorem~\ref{thm-wqo-folklore} (d)).  Thus there are only finitely many distinct classes $\U_G^+$, and from this we can conclude that there are also only finitely many distinct classes $\U_G^-$ (since there is at most one such class for every class $\U_G^+$). Thus we have established that there are only finitely many distinct classes $\U_G$, so there is a finite set $S\subseteq\C$ such that every class $\U_G$ (with $G\in\C$) is of the form $\U_H$ for some $H\in S$.

Define
$$
\U=\bigcup_{H\in S}\U_H.
$$
As each $\U_G$ is a proper subclass of $\C$ and $\C$ is atomic, we know that $\U$ is also a proper subclass of $\C$. We claim that every disconnected graph in $\C$ has a connected component contained in $\U$. To see this, consider a disconnected graph $A\cup B\in\C$. By definition, $A\in\U_B^+$ and $B\in\U_A^+$, so we are done unless both $\U_A^+$ and $\U_B^+$ are equal to $\C$. In this case we see that $\U_A=\U_A^-$ and the disjoint union of $B$ with every graph in $\C$ lies in $\C$, so $B\in\U_A$, as desired.

Recall that every cograph is either a disjoint union or a join. Thus by an analogous argument (or by applying the above argument to the complement class $\overline{\C}$), we see that there is a proper subclass $\J\subsetneq\C$ such that every connected graph in $\C$ contains a join component in $\J$. Moreover, because $\C$ is atomic, we see that $\U\cup\J\subsetneq\C$.

It is now easy to establish by induction on $|V(G)|$ that every graph $G\in\C$ is contained in the inflation of a threshold graph by graphs in $\U\cup\J$.  This fact is clear for $|V(G)|=1$ because $K_1$ must be contained in $\U\cup\J$, and all larger graphs have either a component or a join component in $\U\cup\J$.
\end{proof}

We conclude with a proof of our main theorem.

\newenvironment{proof-thm-lcw-main}{\emph{Proof of Theorem~\ref{thm-lcw-main}.}}{\hfill$\qed$}

\begin{proof-thm-lcw-main}
Suppose that the theorem was false and choose a minimal counterexample, say $\C$, by Theorem~\ref{thm-wqo-folklore} (d). Note that $\C$ must be an infinite class, it cannot contain all quasi-threshold graphs or their complements, and (by the comments before Lemma~\ref{lem-closed-neither}) $\C$ is atomic. 

$\C$ cannot be closed under both disjoint unions and joins, else it contains all cographs. In this case, $\C$ must be contained in the inflation of the threshold graphs by a proper subclass $\D\subsetneq\C$: this follows by Lemma~\ref{lem-closed-one} if $\C$ is closed under disjoint unions or joins (since $\C$ cannot contain all quasi-threshold graphs or their complements), and by Lemma~\ref{lem-closed-neither} if not.

Finally,  Proposition~\ref{prop-inflate-lcw} shows that $\lcw(\C)\le\lcw(\D)+2$, which is finite by our choice of $\C$.
\end{proof-thm-lcw-main}

\minisec{Acknowledgements}
We are grateful to the anonymous referees, whose careful and detailed reports greatly improved the presentation of this paper, and also to Michael Albert,  Jay Pantone, Sang-il Oum, Jisu Jeong and Hojin Choi for their conversations and corrections.

\bibliographystyle{acm}
\bibliography{lcw-cographs}

\end{document}